\documentclass[a4paper,12pt]{article}
\usepackage{times, url}
\usepackage{graphicx,amsmath,amssymb,amsthm,amsfonts,bm,float,cite}\usepackage{mathtools}
\newtheorem{theorem}{Theorem}[section]
\newtheorem{corollary}{Corollary}[section]

\usepackage[pdfencoding=auto]{hyperref}
\usepackage{bookmark}

\usepackage[none]{hyphenat}
\usepackage[center]{caption}
\begin{document}

\begin{center}
{\LARGE \bf  Proving the existence of Euclidean knight's tours \\[4mm] on $n \times n \times \cdots \times n$ chessboards for $n < 4$}
\vspace{12mm}

{\Large \bf Marco Rip\`a}
\vspace{3mm}

World Intelligence Network \\ 
Rome, Italy \\
e-mail: \url{marcokrt1984@yahoo.it}
\vspace{8mm}

\end{center}

\noindent {\bf Abstract:} The Knight's Tour problem consists of finding a Hamiltonian path for the knight on a given set of points so that the knight can visit exactly once every vertex of the mentioned set.\linebreak In the present paper, we provide a $5$-dimensional alternative to the well-known statement that it is not ever possible for a knight to visit once every vertex of $C(3,k)\coloneqq \{0,1,2\}^k$ by performing a sequence of $3^k-1$ jumps of standard length, since the most accurate answer to the original question actually depends on which mathematical assumptions we are making at the beginning of the game, when we decide to extend a planar chess piece to the third dimension and above.
Our counterintuitive outcome follows from the observation that we can alternatively define a $2$D knight as a piece that moves from one square to another on the chessboard by covering a fixed Euclidean distance of $\sqrt{5}$ so that also the statement of Theorem~3 in [Erde, J., Gol{\'e}nia, B., \& Gol{\'e}nia, S. (2012), The closed knight tour problem in higher dimensions, The Electronic Journal of Combinatorics, 19(4), \#P9] does not hold anymore for such a Euclidean knight, as long as a $2 \times 2 \times \cdots \times 2$ chessboard with at least $2^6$ cells is given. Moreover, we show a classical closed knight's tour on $C(3,4)-\{(1,1,1,1)\}$ whose arrival is at a distance of $2$ from $(1,1,1,1)$, and we finally construct closed Euclidean knight's tours on $\{0,1\}^k$ for each integer $k \geq 6$.
\vspace{5mm}


\section{Introduction} \label{sec:Intr}

Given a set $C \subseteq \mathbb{R}^k$ consisting of $m \in \mathbb{Z}^+$ points, it is commonly agreed that every knight's tour is a sequence of $m-1$ knight jumps, of Euclidean length $\sqrt{5}$ chessboard units each (where one chessboard unit is the distance between the centers of adjacent squares of the chessboard), that let the knight visit exactly once all the $m$ vertices of $C$. In particular, we say that the knight's tour is closed if and only if the $m$-th visited vertex of $C$ (including the starting vertex) is at a unit knight-distance from the beginning point, otherwise we have an open knight's tour on $C$.

The origins of the Knight's Tour problem are lost in the centuries, this being a thousands of years old puzzle \cite{18} that lists among its contributors some very big names in mathematics, such as Abraham de Moivre, Alexandre-Th\'eophile Vandermonde, Adrien-Marie Legendre, and Leonard Euler himself, who found one solution for the planar $8 \times 8$ configuration in 1759 \cite{4,8}. Euler's solution is an open knight's tour since the center of the last square visited by the knight is not at a distance of $\sqrt{2^2+1^2}$ chessboard units from the center of its starting square (in the most common sense, for arbitrary $k$, being the beginning vertex at a Euclidean distance of $\sqrt{5}$ from the arrival would represent a necessary but not sufficient condition for having a closed knight's tour).

\sloppy Now, if we agree that the Euclidean $\sqrt{5}$-rule (see \cite{1}, Article~3.6, that uses the superlative of \textit{near} as a criterion for the official knight move rule) defines also the knight metric for any $k$-dimensional $n \times n \times \cdots \times n$ chessboard (while a customizable definition of the discrete knight pattern is at the bottom of the generalized knight's tour problem in two and three dimensions, as described in \cite{12} and \cite{7}, respectively), we trivially have that a $k$-knight is a mathematical object whose move rule consists of performing only jumps having Euclidean length equal to $\sqrt{5}$ chessboard units \cite{6}, from a cell of the given chessboard to one of the remaining $n^k-1$ cells. Thus, we can move our favorite chess piece from the vertex ${\rm{V_1}}$, identified by the $k$-tuple of Cartesian coordinates $(x_1, x_2, \ldots, x_k) : x_1, x_2, \ldots, x_k \in \mathbb\{0,1,\ldots, n-1\}$, to another one, ${\rm{V_2}}\equiv (y_1, y_2, \ldots, y_k)$ also belonging to $C(n,k) \coloneqq \{\underbrace{\{0,1,\ldots, n-1\} \times \{0,1,\ldots, n-1\}\times \cdots \times \{0,1,\ldots, n-1\}}_\textrm{\textit{k}-times}\}$, if and only if
\begin{equation}\label{eq:1}
\sqrt{(y_1-x_1)^2+(y_2-x_2)^2+ \cdots + (y_k-x_k)^2}=\sqrt{5} .
\end{equation}

Hence, (1) can be compactly rewritten as 
\begin{equation}\label{eq:2}
{\sum_{j=1}^{k}(x_j-y_j)^2}=5 .
\end{equation}

It is easily possible to show that, although our \textit{Euclidean knight} produces a metric for every pair $(n,k)$ that allows the usual knight to do so (i.e., $n \geq 4 \wedge k \geq 2$ represents a sufficient condition \cite{21}), it also induces a metric space on any given $3 \times 3 \times \cdots \times 3$ chessboard with at least $3^5$ cells (Section~\ref{sec:2}) and even on every $2 \times 2 \times \cdots \times 2$ chessboard consisting of at least $2^6$ cells (see Section~\ref{sec:4}, Theorem~\ref{Theorem 4.1}).

\sloppy Furthermore, we construct an open Euclidean knight's tour on $\{\{0,1,2\} \times \{0,1,2\}\times \{0,1,2\}\times \{0,1,2\} \times \{0,1,2\}\}\subseteq \mathbb{Z}^5$ and provide a closed Euclidean knight's tour on $\{\{0,1\} \times \{0,1\} \times \{0,1\} \times \{0,1\} \times \{0,1\} \times \{0,1\}\}\subseteq \mathbb{Z}^6$ (which can be smoothly generalized, for any given $k \geq 6$, to $\{\{0,1\} \times \{0,1\} \times \cdots \times \{0,1\}\}\subseteq \mathbb{Z}^k$).

Accordingly, Section~\ref{sec:2} is devoted to proving the existence of Euclidean knight's tours when $n=3$ is given, and then a pair of corollaries will follow, while Section~\ref{sec:3} proposes a variation of the main problem \cite{5} by removing the central vertex from any grid $\{\{0,1,2\} \times \{0,1,2\}\times \cdots \times \{0,1,2\}\}\subseteq \mathbb{Z}^k$ \cite{11}, under the additional constraint of ending the path in a vertex that is at a Euclidean distance of $\sqrt{k}$ from the missing point $(1,1,\ldots,1)$ (for the related problem of determining the existence of closed, conventional, knight's tours on boxes, see Theorem~1 of \cite{9}).

Finally, Section~\ref{sec:4} entirely covers the $n=2$ case.


\section{Euclidean knight's tours on a \texorpdfstring{$3 \times 3 \times 3 \times 3 \times 3$}{3 x 3 x 3 x 3 x 3} chessboard} \label{sec:2}

Here, we consider the problem of finding (possibly open) Euclidean knight's tours on $3 \times 3 \times \cdots \times 3$ chessboards. Then, in order to constructively show that a Euclidean knight's tour exists only if $k > 4$, it is sufficient to point out that the (Euclidean) distance between the vertex $(1, 1, \ldots, 1)$ and one of the farthest $2^k$ vertices (\textit{corners}) of the $k$-cube $\{[0,2] \times [0,2] \times \cdots \times [0,2]\} \subseteq \mathbb{R}^k$ (i.e., the distance between $(1, 1, \ldots, 1)$ and any element of $\{\{0,2\}\times \{0,2\} \times \cdots \times \{0,2\}\}\subseteq \mathbb{Z}^k$) is equal to $\sqrt{k}$, for any positive integer $k$.

Hence, starting at $(1, 1, \ldots, 1)$, by Equation (\ref{eq:2}), $5$ is the minimum value of $k$ such that our Euclidean knight can make one single move on the $3 \times 3 \times \cdots \times 3$ grid.

In this regard, let us observe how Qing and Watkins indicated a different way to extend in $3$D the planar knight's move pattern by proposing, in \cite{9}, pp. 45--46, that every knight jump has to mandatorily change all the $k=3$ Cartesian coordinates of its starting vertex (by $2^0$, $2^1$, and $2^2$). Although this personal interpretation of Article~3.6 of \cite{1} is obviously not compatible with the existence of any knight's tour on $3 \times 3 \times \cdots \times 3$ grids (since the Euclidean distance between $(1,1,\ldots, 1)$ and $(0,0,\ldots, 0)$ is equal to $\sqrt{k}$, which is clearly smaller than $\sqrt{{{(2^0)}^2}+{{(2^1)}^2}+{{(2^2)}^2}+\cdots+{(2^{k-1})}^2}=\sqrt{\sum_{j=0}^{k-1}4^j}=\frac{\sqrt{4^k-1}}{\sqrt{3}}$ for any $k>1$), the underlying idea of a $k$-knight that can change (or has to mandatorily change) the values of all its $k$ Cartesian coordinates by performing a single move is fascinating and useful \cite{19}.

Now we are ready to prove that an open Euclidean knight's tour actually exists if $n=3$ and $k$ is set at $5$, so the trivial consideration that the knight graph is not connected for any $3 \times 3 \times \cdots \times 3$ board does not apply anymore, as the $k$-knight is a Euclidean $k$-knight.

\begin{theorem} \label{Theorem 2.1}
Let $h \in \{0,1,2,\ldots,3^k-1\}$ and assume that the knight move rule from the vertex ${\rm{V}}_h \equiv (x_1,x_2, \ldots, x_k)$ to the next vertex, ${\rm{V}}_{h+1} \equiv (y_1,y_2, \ldots, y_k)$, of $C(3,k) \coloneqq \{0,1,2\}^k$ is given by $d({{\rm{V}}}_h, {\rm{V}}_{h+1}) \coloneqq \sqrt{\sum_{j=1}^{k}(x_j-y_j)^2}=\sqrt{5}$. Then, the minimum value of $k$ that produces a knight's tour on $C(3,k)$ is $5$.
\end{theorem}

\begin{proof}
As we have already observed, $d((0,0,\ldots,0), (1,1,\ldots,1))=\sqrt{k}$ (for any $k \in \mathbb{Z}^+$) and this implies that $d((0,0,\ldots,0), (1,1,\ldots,1)) < \sqrt{5}$ if and only if $k < 5$. Thus, $k$ cannot be less than $5$.

Accordingly, let us constructively prove Theorem~\ref{Theorem 2.1} by simply providing the sequence of the $243$ Cartesian coordinates that describe an open knight's tour on the set $C(3,5)$ (see Figure~\ref{fig:Figure_1} for a graphical proof), since every vertex is visited exactly once and the Euclidean length of each knight jump is equal to $\sqrt{5}$.

Then, the polygonal chain $P_o(3,5) \coloneqq (1,0,2,0,1)\rightarrow(2,0,2,2,1)\rightarrow(2,2,2,2,0)\rightarrow(1,0,2,2,0)\rightarrow(1,1,0,2,0)\rightarrow(0,1,2,2,0)\rightarrow(2,0,2,2,0)\rightarrow(1,2,2,2,0)\rightarrow(0,0,2,2,0)\rightarrow(2,1,2,2,0)\rightarrow(0,2,2,2,0)\rightarrow(0,0,1,2,0)\rightarrow(1,2,1,2,0)\rightarrow(1,1,1,0,0)\rightarrow(2,1,1,2,0)\rightarrow(0,2,1,2,0)\rightarrow(1,0,1,2,0)\rightarrow(2,2,1,2,0)\rightarrow(0,1,1,2,0)\rightarrow(2,0,1,2,0)\rightarrow(2,2,0,2,0)\rightarrow(1,0,0,2,0)\rightarrow(1,1,2,2,0)\rightarrow(0,1,0,2,0)\rightarrow(2,0,0,2,0)\rightarrow(1,2,0,2,0)\rightarrow(0,0,0,2,0)\rightarrow(2,1,0,2,0)\rightarrow(0,2,0,2,0)\rightarrow(0,0,0,1,0)\rightarrow(1,2,0,1,0)\rightarrow(1,1,2,1,0)\rightarrow(2,1,0,1,0)\rightarrow(0,2,0,1,0)\rightarrow(1,0,0,1,0)\rightarrow(2,2,0,1,0)\rightarrow(0,1,0,1,0)\rightarrow(2,0,0,1,0)\rightarrow(2,2,1,1,0)\rightarrow(1,0,1,1,0)\rightarrow(1,1,1,1,2)\rightarrow(0,1,1,1,0)\rightarrow(2,0,1,1,0)\rightarrow(1,2,1,1,0)\rightarrow(0,0,1,1,0)\rightarrow(2,1,1,1,0)\rightarrow(0,2,1,1,0)\rightarrow(0,0,2,1,0)\rightarrow(1,2,2,1,0)\rightarrow(1,1,0,1,0)\rightarrow(2,1,2,1,0)\rightarrow(0,2,2,1,0)\rightarrow(1,0,2,1,0)\rightarrow(2,2,2,1,0)\rightarrow(0,1,2,1,0)\rightarrow(2,0,2,1,0)\rightarrow(2,2,2,0,0)\rightarrow(1,0,2,0,0)\rightarrow(1,1,0,0,0)\rightarrow(0,1,2,0,0)\rightarrow(2,0,2,0,0)\rightarrow(1,2,2,0,0)\rightarrow(0,0,2,0,0)\rightarrow(2,1,2,0,0)\rightarrow(0,2,2,0,0)\rightarrow(0,0,1,0,0)\rightarrow(1,2,1,0,0)\rightarrow(1,1,1,2,0)\rightarrow(2,1,1,0,0)\rightarrow(0,2,1,0,0)\rightarrow(1,0,1,0,0)\rightarrow(2,2,1,0,0)\rightarrow(0,1,1,0,0)\rightarrow(2,0,1,0,0)\rightarrow(2,2,0,0,0)\rightarrow(1,0,0,0,0)\rightarrow(1,1,2,0,0)\rightarrow(0,1,0,0,0)\rightarrow(2,0,0,0,0)\rightarrow(1,2,0,0,0)\rightarrow(0,0,0,0,0)\rightarrow(2,1,0,0,0)\rightarrow \underline{(0,2,0,0,0)} \rightarrow \underline{(1,1,1,1,1)}\rightarrow\underline{(0,2,0,0,2)} \rightarrow(2,1,0,0,2)\rightarrow(0,0,0,0,2)\rightarrow(1,2,0,0,2)\rightarrow(2,0,0,0,2)\rightarrow(0,1,0,0,2)\rightarrow(1,1,2,0,2)\rightarrow(1,0,0,0,2)\rightarrow(2,2,0,0,2)\rightarrow(2,0,1,0,2)\rightarrow(0,1,1,0,2)\rightarrow(2,2,1,0,2)\rightarrow(1,0,1,0,2)\rightarrow(0,2,1,0,2)\rightarrow(2,1,1,0,2)\rightarrow(1,1,1,2,2)\rightarrow(1,2,1,0,2)\rightarrow(0,0,1,0,2)\rightarrow(0,2,2,0,2)\rightarrow(2,1,2,0,2)\rightarrow(0,0,2,0,2)\rightarrow(1,2,2,0,2)\rightarrow(2,0,2,0,2)\rightarrow(0,1,2,0,2)\rightarrow(1,1,0,0,2)\rightarrow(1,0,2,0,2)\rightarrow(2,2,2,0,2)\rightarrow(2,0,2,1,2)\rightarrow(0,1,2,1,2)\rightarrow(2,2,2,1,2)\rightarrow(1,0,2,1,2)\rightarrow(0,2,2,1,2)\rightarrow(2,1,2,1,2)\rightarrow(1,1,0,1,2)\rightarrow(1,2,2,1,2)\rightarrow(0,0,2,1,2)\rightarrow(0,2,1,1,2)\rightarrow(2,1,1,1,2)\rightarrow(0,0,1,1,2)\rightarrow(1,2,1,1,2)\rightarrow(2,0,1,1,2)\rightarrow(0,1,1,1,2)\rightarrow(1,1,1,1,0)\rightarrow(1,0,1,1,2)\rightarrow(2,2,1,1,2)\rightarrow(2,0,0,1,2)\rightarrow(0,1,0,1,2)\rightarrow(2,2,0,1,2)\rightarrow(1,0,0,1,2)\rightarrow(0,2,0,1,2)\rightarrow(2,1,0,1,2)\rightarrow(1,1,2,1,2)\rightarrow(1,2,0,1,2)\rightarrow(0,0,0,1,2)\rightarrow(0,2,0,2,2)\rightarrow(2,1,0,2,2)\rightarrow(0,0,0,2,2)\rightarrow(1,2,0,2,2)\rightarrow(2,0,0,2,2)\rightarrow(0,1,0,2,2)\rightarrow(1,1,2,2,2)\rightarrow(1,0,0,2,2)\rightarrow(2,2,0,2,2)\rightarrow(2,0,1,2,2)\rightarrow(0,1,1,2,2)\rightarrow(2,2,1,2,2)\rightarrow(1,0,1,2,2)\rightarrow(0,2,1,2,2)\rightarrow(2,1,1,2,2)\rightarrow(1,1,1,0,2)\rightarrow(1,2,1,2,2)\rightarrow(0,0,1,2,2)\rightarrow(0,2,2,2,2)\rightarrow(2,1,2,2,2)\rightarrow(0,0,2,2,2)\rightarrow(1,2,2,2,2)\rightarrow(2,0,2,2,2)\rightarrow(0,1,2,2,2)\rightarrow(1,1,0,2,2)\rightarrow(1,0,2,2,2)\rightarrow(2,2,2,2,2)\rightarrow(0,2,2,2,1)\rightarrow(1,0,2,2,1)\rightarrow(2,2,2,2,1)\rightarrow(0,1,2,2,1)\rightarrow(1,1,0,2,1)\rightarrow(1,2,2,2,1)\rightarrow(1,1,2,0,1)\rightarrow(2,1,2,2,1)\rightarrow(0,0,2,2,1)\rightarrow(2,0,1,2,1)\rightarrow(1,2,1,2,1)\rightarrow(0,0,1,2,1)\rightarrow(2,1,1,2,1)\rightarrow(0,2,1,2,1)\rightarrow(1,0,1,2,1)\rightarrow(1,1,1,0,1)\rightarrow(0,1,1,2,1)\rightarrow(2,2,1,2,1)\rightarrow(2,0,0,2,1)\rightarrow(1,2,0,2,1)\rightarrow(0,0,0,2,1)\rightarrow(2,1,0,2,1)\rightarrow(0,2,0,2,1)\rightarrow(1,0,0,2,1)\rightarrow(1,1,2,2,1)\rightarrow(0,1,0,2,1)\rightarrow(2,2,0,2,1)\rightarrow(2,0,0,1,1)\rightarrow(1,2,0,1,1)\rightarrow(0,0,0,1,1)\rightarrow(2,1,0,1,1)\rightarrow(0,2,0,1,1)\rightarrow(1,0,0,1,1)\rightarrow(1,1,2,1,1)\rightarrow(0,1,0,1,1)\rightarrow(2,2,0,1,1)\rightarrow(2,0,0,0,1)\rightarrow(0,1,0,0,1)\rightarrow(2,2,0,0,1)\rightarrow(1,0,0,0,1)\rightarrow(0,2,0,0,1)\rightarrow(2,1,0,0,1)\rightarrow(0,0,0,0,1)\rightarrow(1,2,0,0,1)\rightarrow(1,0,1,0,1)\rightarrow(0,2,1,0,1)\rightarrow(2,1,1,0,1)\rightarrow(0,0,1,0,1)\rightarrow(1,2,1,0,1)\rightarrow(1,1,1,2,1)\rightarrow(0,1,1,0,1)\rightarrow(2,2,1,0,1)\rightarrow(2,0,1,1,1)\rightarrow(1,2,1,1,1)\rightarrow(0,0,1,1,1)\rightarrow(2,1,1,1,1)\rightarrow(0,2,1,1,1)\rightarrow(1,0,1,1,1)\rightarrow(2,2,1,1,1)\rightarrow(0,1,1,1,1)\rightarrow(2,1,2,1,1)\rightarrow(0,2,2,1,1)\rightarrow(1,0,2,1,1)\rightarrow(2,2,2,1,1)\rightarrow(0,1,2,1,1)\rightarrow(1,1,0,1,1)\rightarrow(1,2,2,1,1)\rightarrow(2,0,2,1,1)\rightarrow(0,0,2,0,1)\rightarrow(2,0,1,0,1)\rightarrow(2,2,2,0,1)\rightarrow(0,1,2,0,1)\rightarrow(2,0,2,0,1)\rightarrow(0,0,2,1,1)\rightarrow(0,2,2,0,1)\rightarrow(2,1,2,0,1)\rightarrow(1,1,0,0,1)\rightarrow(1,2,2,0,1)$ has link length $242$ and represents a knight's tour for the given set (let us highlight that also $d({{\rm{V}}}_{82}, {{\rm{V}}}_{83})=d({{\rm{V}}}_{83}, {{\rm{V}}}_{84})=\sqrt{5}$, even if this time the taxicab length \cite{2,3} of our knight jump is equal to $5$, instead of $3$ as for any other knight move), while the Euclidean distance between its starting point and ending point is $2$, so the knight tour is open.

\begin{figure}[H]
\begin{center}
\includegraphics[width=\linewidth]{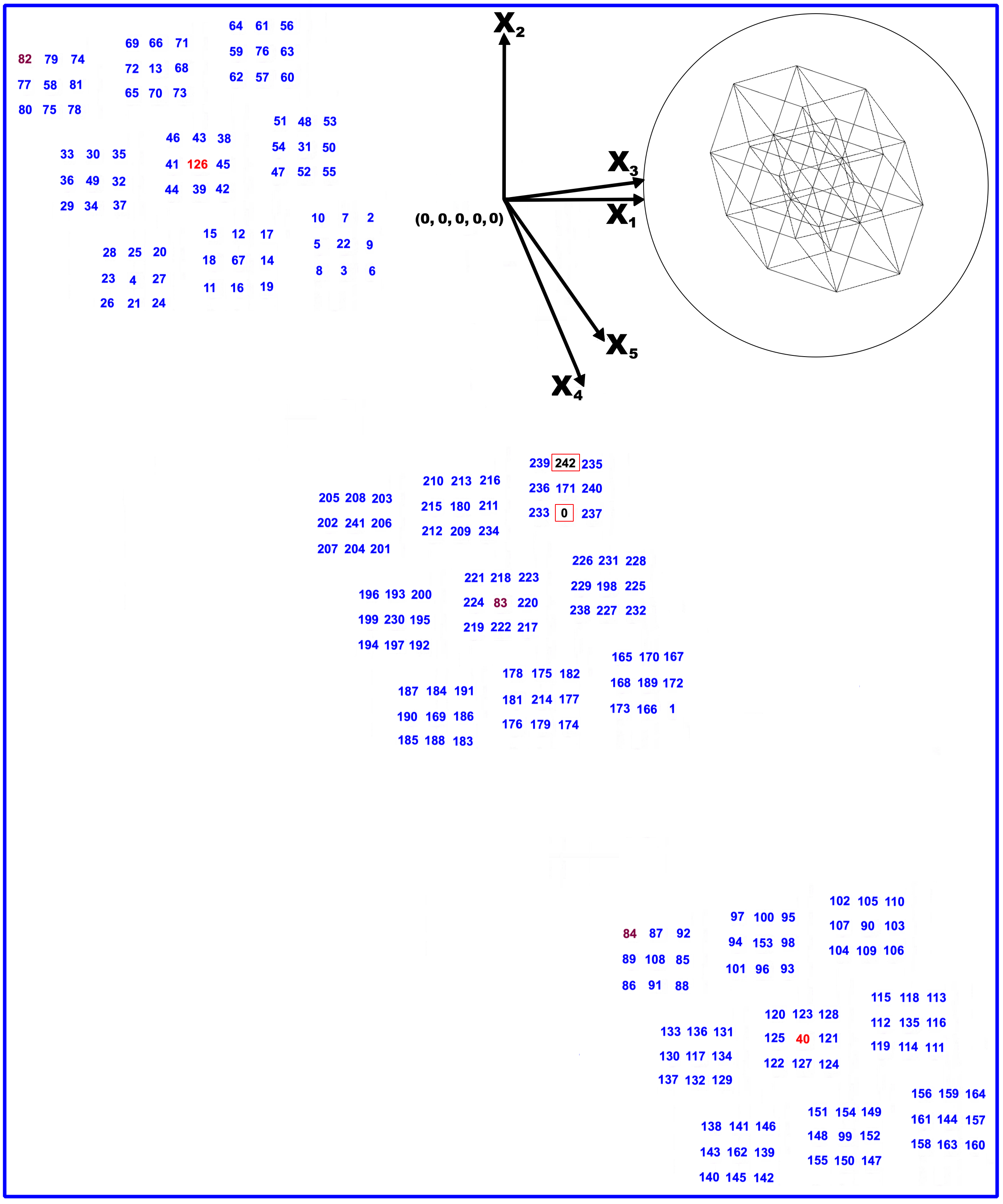}
\end{center}
\caption{A graphical representation of an open Euclidean knight's tour on $C(3,5) \coloneqq \{\{0,1,2\} \times \{0,1,2\}\times \{0,1,2\}\times \{0,1,2\} \times \{0,1,2\}\}$.}
\label{fig:Figure_1}
\end{figure}

\vspace{8mm}

Therefore, a knight tour exists for the set $\{\{0,1,2\} \times \{0,1,2\} \times \{0,1,2\} \times \{0,1,2\} \times \{0,1,2\}\} \subseteq \mathbb{Z}^5$, while it is impossible to achieve it on $C(3,k)$ if $k < 5$, and this concludes the proof of the theorem.
\end{proof}

\begin{corollary} \label{Corollary 2.2}
For any given $k \in \mathbb{N}-\{0,1\}$, it does not exist any closed Euclidean knight's tour on $C(3,k)$.
\end{corollary}

\begin{proof}
In order to prove Corollary~\ref{Corollary 2.2}, let us introduce the well-known parity argument \cite{9}.

Then, we have that the generic vertex $(x_1,x_2,\ldots,x_k)$ of $\{0,1,2,\ldots, n-1\}^k$ is a dark vertex if and only if $\sum_{j=1}^{k} x_j \equiv 0 \pmod{2}$, whereas any light vertex is such that $\sum_{j=1}^{k} x_j \equiv 1 \pmod{2}$ (i.e., if we add together the $k$ Cartesian coordinates of a dark vertex of $C$ the result is always an even number, otherwise the given vertex is a light vertex).
Figure~\ref{fig:Figure_3} shows how to consistently represent the set $C(3,5)$ as a $5$D chessboard.

\begin{figure}[H]
\begin{center}
\includegraphics[width=\linewidth]{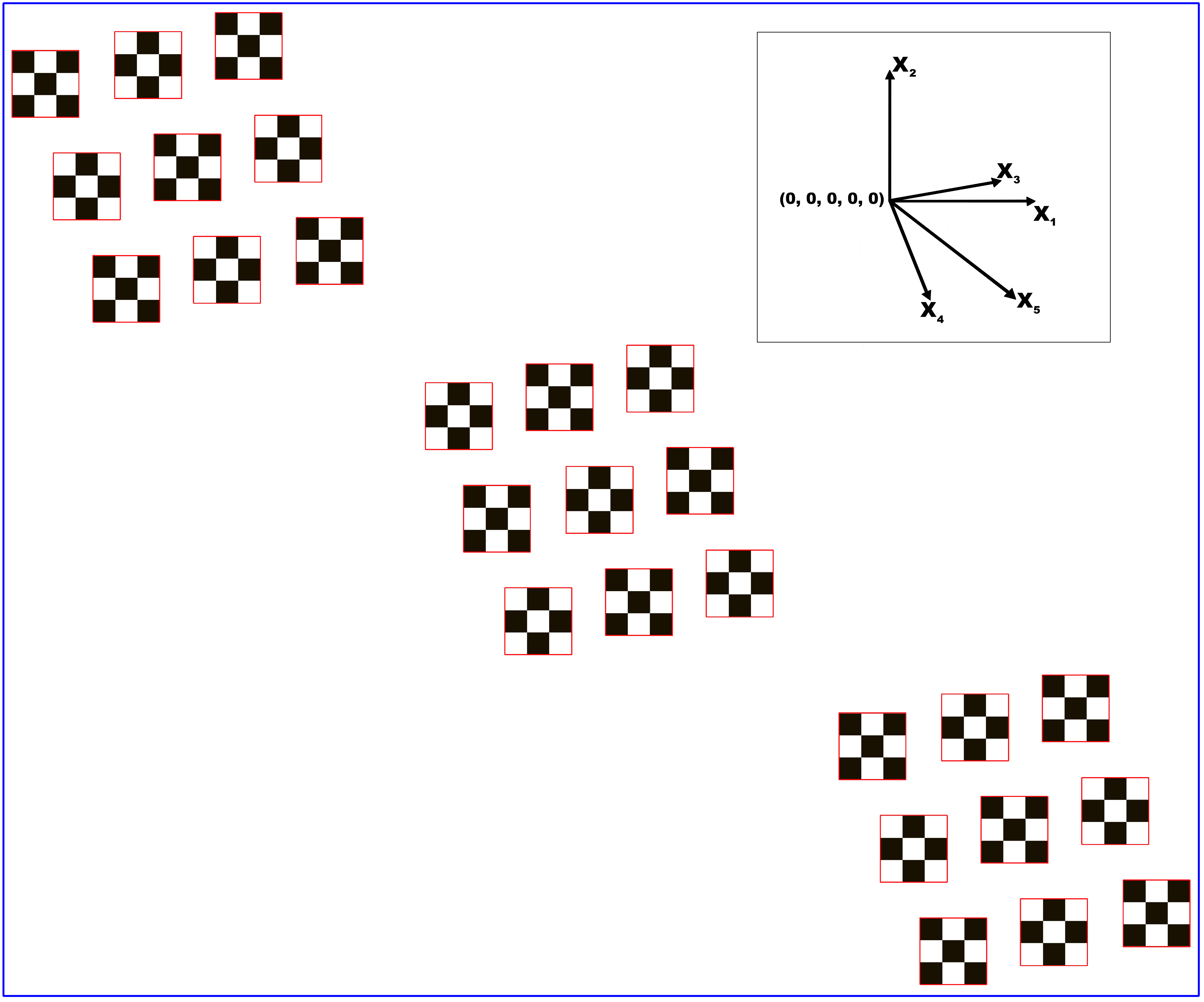}
\end{center}
\caption{Coloring the $3 \times 3 \times 3 \times 3 \times 3$ chessboard in a proper way \linebreak (thanks to the parity argument).}
\label{fig:Figure_3}
\end{figure}

\newpage Thus, we invoke the parity argument by observing that the knight (including the Euclidean knight, which is subject to the same constraint by construction since $\left | 1 \right |+\left | 1 \right |+\left | 1 \right |+\left | 1 \right |+\left | 1 \right |$ is odd as $\left | 1 \right | + \left | 2 \right |$) can only move from a dark square to a light square and vice versa (see \cite{20}, Figure~2).
For this purpose, let us observe that the taxicab length \cite{3} of any Euclidean knight jump has to necessarily be $3$ or $5$ since all the Diophantine equations of the form $5=t_1^2+t_2^2+\cdots +t_k^2$ (see Equation (\ref{eq:2})) admit only two types of solutions, one having only $2$ non-zero terms (i.e., a $\pm 1$ term and a $\pm 2$ term) and a second set of solutions which is characterized by $5$ non-zero terms that are all elements of $\{-1, 1\}$.

Hence, a necessary (but not sufficient) condition for having a closed (possibly Euclidean) knight's tour on $C(n,k)$ is that the number of light vertices is equal to the number of dark vertices \cite{9}, and this is impossible if $n=3$ since $3^k$ is odd for any $k$.

Therefore, every $3 \times 3 \times \cdots \times 3$ chessboard does not admit any closed Euclidean knight's tour.
\end{proof}

\begin{corollary} \label{Corollary 2.3}
A Euclidean knight can visit every vertex of $C(3,5)$ and then return to the vertex on which it began by performing no more than $3^5+1$ jumps.
\end{corollary}

\begin{proof}
It is sufficient to observe that we can close the polygonal chain $P_o(3,5)$ (see proof of Theorem \ref{Theorem 2.1} and Figure \ref{fig:Figure_1}) by visiting the vertex $(1,1,0,0,1)$ on move $243$, finally reaching the starting vertex, $(1,0,2,0,1)$, on move $244$ (i.e., we visit twice the vertex $(1,1,0,0,1)$). Then, the statement of Corollary~\ref{Corollary 2.3} trivially follows.
\end{proof}


\section{Closed knight's tours on \texorpdfstring{$C(3,k)-\{(1,1,\ldots,1)\}$}{C(3,k)-{1,1,...,1}}} \label{sec:3}

In this section, we introduce the problem of finding closed (possibly Euclidean) knight's tours on given $k$-dimensional grids of the form $\{\{0,1,2\} \times \{0,1,2\} \times \cdots \times \{0,1,2\}\}-\{(1,1,\ldots,1)\}$ (for planar knight's tours on rectangular chessboards with holes, see \cite{10} and \cite{11}).

Then, we constructively show that a closed regular knight's tour certainly exists on $C(3,k)-{(1,1,\ldots,1)}$ if $k \in \{2,4\}$ (i.e., here we consider the usual $k$D grids of rank $3$ without their central vertex, together with the classical knight that can only make L-shaped moves of taxicab length $3$), where the existence of a closed (regular) knight's tour represents a sufficient condition for the existence of an open knight's tour on the same set of vertices, even under the constraint of having the arrival of the polygonal chain at a Euclidean distance of $\sqrt{k}$ from $(1,1,\ldots, 1)$.

Since any closed/open Euclidean knight's tour is also a closed/open regular knight's tour as long as $k<5$, for the sake of simplicity, we take into account here only the sets $C(3,2)-\{(1,1)\}$, $C(3,3)-\{(1,1,1)\}$, and $C(3,4)-\{(1,1,1,1)\}$, showing that a closed knight's tour exists for the aforementioned $2$D and $4$D cases.

Thus, if $k=2$, then a valid solution to the stated problem is given by the trivial polygonal chain $P_c(3,2;\{(1,1)\}) \coloneqq (2,1)\rightarrow(0,2)\rightarrow(1,0)\rightarrow(2,2)\rightarrow(0,1)\rightarrow(2,0)\rightarrow(1,2)\rightarrow(0,0)$ (see Figure~\ref{fig:Figure_2}, and also Figure~1 of Reference \cite{11} for the related cycle $P_c(3,2;\{(1,1)\}) \cup \{(0,0)\rightarrow(2,1)\}$).

If $k=3$, the best achievable sequence of knight jumps has length $24$ (see \cite{13}, page 66, Figures 4\&5).

In addition, Figure~\ref{fig:Figure_2} shows another polygonal chain, $\overline{P}(3,3;\{(1,1,1),(2,0,0)\}) \coloneqq (0,0,0)\rightarrow(1,2,0)\rightarrow(1,1,2)\rightarrow(0,1,0)\rightarrow(2,2,0)\rightarrow(1,0,0)\rightarrow(0,2,0)\rightarrow(2,1,0)\rightarrow(0,1,1)\rightarrow(2,2,1)\rightarrow(1,0,1)\rightarrow(0,2,1)\rightarrow(2,1,1)\rightarrow(0,0,1)\rightarrow(1,2,1)\rightarrow(2,0,0)\rightarrow(2,2,2)\rightarrow(1,0,2)\rightarrow(0,2,2)\rightarrow(2,1,2)\rightarrow(0,0,2)\rightarrow(1,2,2)\rightarrow(1,1,0)\rightarrow(0,1,2)\rightarrow(2,0,2)$, consisting of $25$ nodes and $24$ links too.
Then, by looking at the arrival of $\overline{P}(3,3;\{(1,1,1),(2,0,0)\})$, it is immediate to note that we would have got an open Euclidean knight's tour on $C(3,3)$ if the distance between the corners of $\{\{0,1,2\}\times \{0,1,2\}\times \{0,1,2\}\}$ and the center of the same grid would have been equal to $\sqrt{5}$, instead of $\sqrt{3}$.

\begin{figure}[H]
\begin{center}
\includegraphics[width=\linewidth]{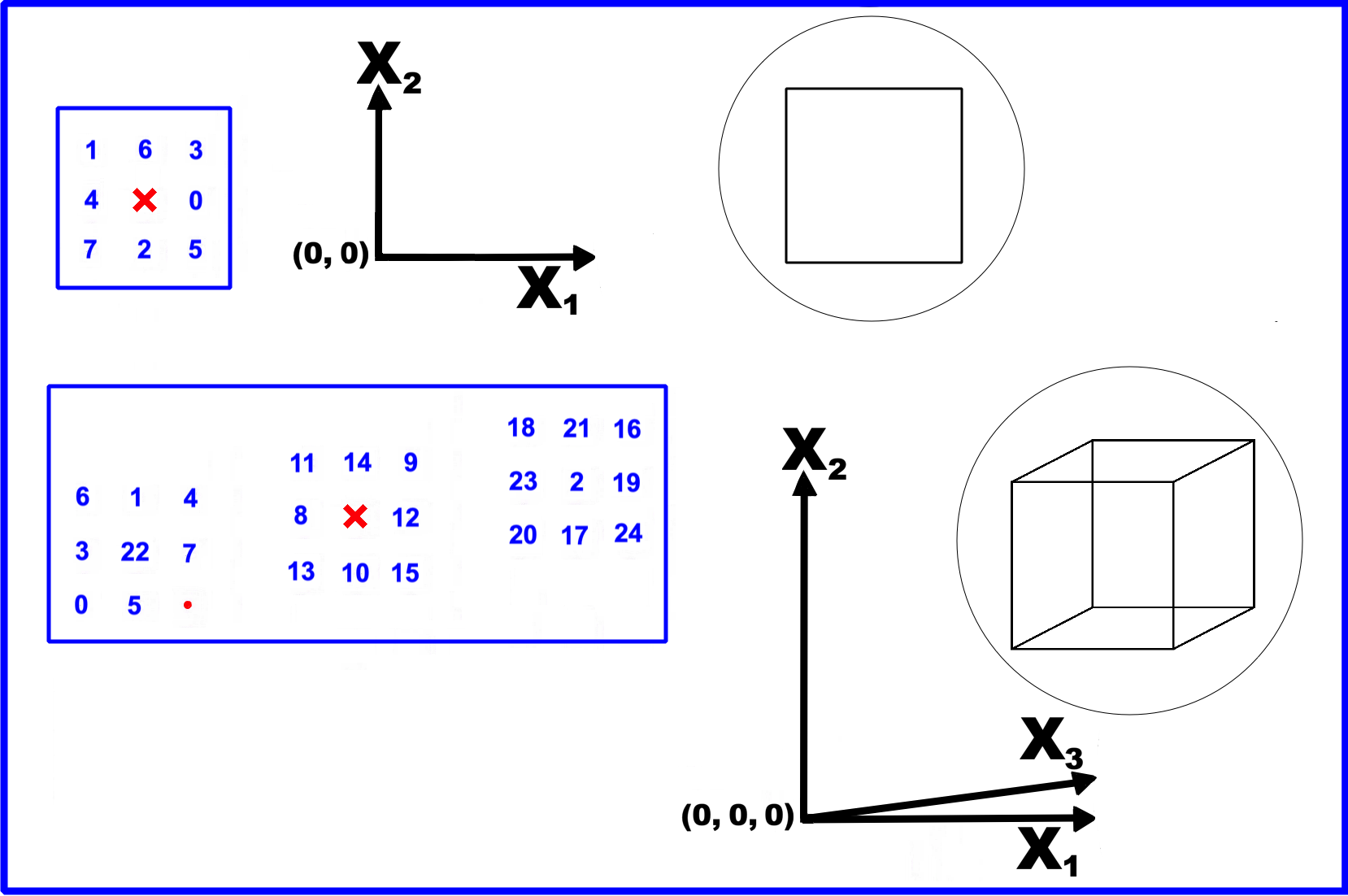}
\end{center}
\caption{The polygonal chains $P_c(3,2;\{(1,1)\})$ and $\overline{P}(3,3;\{(1,1,1),(2,0,0)\})$ visit all the vertices of $C(3,2)-\{(1,1)\}$ and $C(3,3)-\{(1,1,1), (2,0,0)\}$, respectively. Both of them follow the regular (and thus also the Euclidean) knight move rule and end at the corners of $C$.}
\label{fig:Figure_2}
\end{figure}

Now, let $k=4$.
Then, it is certainly possible to achieve closed knight's tours on $C(3,4)-\{(1,1,1,1)\}$, as shown by the Hamiltonian cycle $P_c(3,4;\{(1,1,1,1)\}) \cup \{(0,2,0,1) \rightarrow(0,0,0,2)\}$ (see Figure~\ref{fig:Figure_4}), where $P_c(3,4;\{(1,1,1,1)\}) \coloneqq (0,0,0,2)\rightarrow(2,1,0,2)\rightarrow(0,2,0,2)\rightarrow(1,0,0,2)\rightarrow(2,2,0,2)\rightarrow(0,1,0,2)\rightarrow(1,1,2,2)\rightarrow(1,2,0,2)\rightarrow(2,0,0,2)\rightarrow(0,0,1,2)\rightarrow(1,2,1,2)\rightarrow(2,0,1,2)\rightarrow(0,1,1,2)\rightarrow(2,2,1,2)\rightarrow(1,0,1,2)\rightarrow(1,1,1,0)\rightarrow(2,1,1,2)\rightarrow(0,2,1,2)\rightarrow(0,0,2,2)\rightarrow(1,2,2,2)\rightarrow(2,0,2,2)\rightarrow(0,1,2,2)\rightarrow(2,2,2,2)\rightarrow(1,0,2,2)\rightarrow(1,1,2,0)\rightarrow(2,1,2,2)\rightarrow(0,2,2,2)\rightarrow(0,0,2,1)\rightarrow(1,2,2,1)\rightarrow(2,0,2,1)\rightarrow(0,1,2,1)\rightarrow(2,2,2,1)\rightarrow(1,0,2,1)\rightarrow(1,1,0,1)\rightarrow(2,1,2,1)\rightarrow(0,2,2,1)\rightarrow(0,0,2,0)\rightarrow(1,2,2,0)\rightarrow(2,0,2,0)\rightarrow(0,1,2,0)\rightarrow(2,2,2,0)\rightarrow(1,0,2,0)\rightarrow(1,1,0,0)\rightarrow(2,1,2,0)\rightarrow(0,2,2,0)\rightarrow(0,0,1,0)\rightarrow(2,1,1,0)\rightarrow(0,2,1,0)\rightarrow(1,0,1,0)\rightarrow(2,2,1,0)\rightarrow(0,1,1,0)\rightarrow(1,1,1,2)\rightarrow(1,2,1,0)\rightarrow(2,0,1,0)\rightarrow(0,0,1,1)\rightarrow(1,2,1,1)\rightarrow(2,0,1,1)\rightarrow(0,1,1,1)\rightarrow(2,2,1,1)\rightarrow(1,0,1,1)\rightarrow(0,2,1,1)\rightarrow(2,1,1,1)\rightarrow(0,1,0,1)\rightarrow(2,0,0,1)\rightarrow(1,2,0,1)\rightarrow(0,0,0,1)\rightarrow(2,1,0,1)\rightarrow(1,1,2,1)\rightarrow(1,0,0,1)\rightarrow(2,2,0,1)\rightarrow(2,0,0,0)\rightarrow(1,2,0,0)\rightarrow(0,0,0,0)\rightarrow(2,1,0,0)\rightarrow(0,2,0,0)\rightarrow(1,0,0,0)\rightarrow(1,1,0,2)\rightarrow(0,1,0,0)\rightarrow(2,2,0,0)\rightarrow(0,2,0,1)$.

\begin{figure}[H]
\begin{center}
\includegraphics[width=\linewidth]{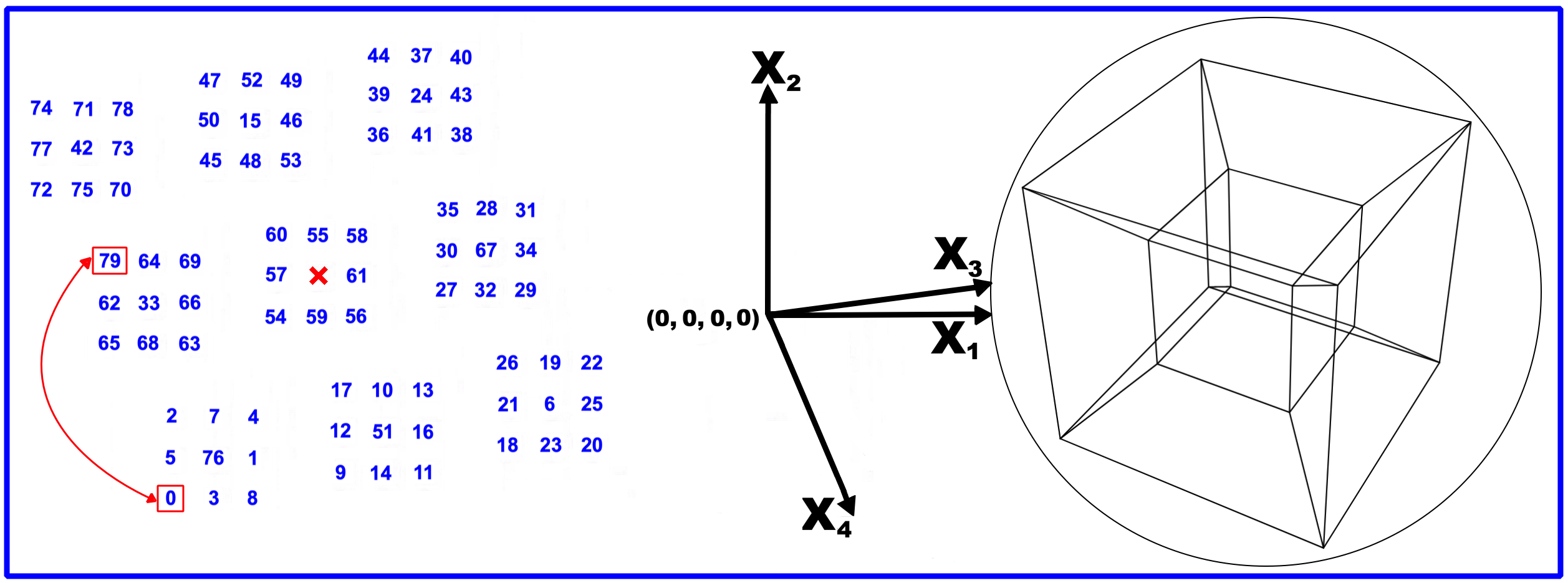}
\end{center}
\caption{A graphical representation of the closed (regular) knight's tour $P_c(3,4;\{(1,1,1,1)\})$ on $C(3,4)-\{(1,1,1,1)\}$.}
\label{fig:Figure_4}
\end{figure}

Thus, we have constructively proven the existence of a closed knight's tour also for \linebreak $C(3,4)-\{(1,1,1,1)\}$, and we can conjecture the existence of closed Euclidean knight's tours on $C(3,k)-\{(1,1,\ldots,1)\}$ for each $k \in \mathbb{Z}^+ : k \equiv 0 \pmod{2}$.


\section{Closed Euclidean knight's tours on any \texorpdfstring{$C(2,k) : k \geq 6$}{C(2,k) : k >= 6}} \label{sec:4}

This section is devoted to providing a counterexample to Theorem~3 of Reference \cite{17} by assuming that the knight is also a Euclidean knight, as defined in Section~\ref{sec:Intr}. 
Furthermore, we prove a more general theorem on the existence of closed Euclidean knight's tours for any ($k$-dimensional) metric space $C(2,k)$ \cite{21} such that $k \geq 6$.

In detail, the mentioned Theorem~3 assumes $k \geq 3$ and states that a $k$-dimensional rectangular grid of the form $n_1 \times n_2 \times \cdots \times n_k$, such that $2 \leq n_1 \leq n_2 \leq \cdots \leq n_k$, admits a closed knight's tour if and only if the following three conditions hold:

\begin{enumerate}
\item $\prod_{j=1}^{k} n_j \equiv 0 \pmod 2$,
\item $n_{k-1} \geq 3$,
\item $n_{k} \geq 4$.
\end{enumerate}

Thus, if $n_1 = n_2 = \cdots = n_k = 2$, Theorem~3 of Reference \cite{17} would imply that $C(2,k)$ does not admit any closed knight tour, but this is no longer true if closed Euclidean knight's tours are included.

\begin{theorem} \label{Theorem 4.1}
Euclidean knight's tours exist on a $\underbrace{2 \times 2 \times \cdots \times 2}_\textrm{\textit{k}-times}$ chessboard if and only if $k \geq 6$. Furthermore, if there is a Euclidean knight's tour on a given $2 \times 2 \times \cdots \times 2$ chessboard, then closed Euclidean knight's tours exist for the same chessboard.
\end{theorem}

\begin{proof}
In order to state the necessary and sufficient condition for the existence of Euclidean knight's tours on $2 \times 2 \times \cdots \times 2$ chessboards, let us preliminary point out that a Euclidean knight can move also on a $2 \times 2 \times \cdots \times 2$ chessboard, but this is possible only if $k\geq 5$, since here is mandatory that a Euclidean knight changes by one $5$ Cartesian coordinates at any jump, given the fact that $n < 3$ does not let the knight make its well-known L-shaped move.
Additionally, it is trivial to observe that only one move is possible on $C(2,5)$ since the maximum distance between two vertices of a $k$-cube corresponds to the distance between opposite pairs of corners, which is always equal to $\sqrt{k}$. Consequently, given $x_1, x_2, x_3, x_4, x_5 \in \{0,1\}$, if the (Euclidean) knight starts at $(x_1, x_2, x_3, x_4, x_5)$, it can only jump to $(|x_1-1|, |x_2-1|, |x_3-1|, |x_4-1|, |x_5-1|)$ and then none of the unvisited vertices of $C(2,5)$ is far enough to allow a second move.

Accordingly, we prove Theorem \ref{Theorem 4.1} by providing the Hamiltonian cycle $P_c{(2,6)} \cup \{(0,1,1,1,1,1) \rightarrow (0,0,0,0,0,0)\}$ (see below) for $C(2,6)$ in order to show the existence of a closed Euclidean knight's tour on the aforementioned set of $64$ vertices. Then, it will be easy to see that closed Euclidean knight's tours exist for any other $C(2,k)$ such that $k \geq 7$.

In detail, the polygonal chain $P_c{(2,6)} \coloneqq (0,0,0,0,0,0)\rightarrow
(1,1,1,1,1,0)\rightarrow
(0,0,0,0,1,1)\rightarrow
(1,1,1,1,0,1)\rightarrow
(0,0,0,1,1,0)\rightarrow
(1,1,1,0,0,0)\rightarrow
(0,0,0,1,0,1)\rightarrow
(1,1,1,0,1,1)\rightarrow
(0,0,1,1,0,0)\rightarrow
(1,1,0,0,1,0)\rightarrow
(0,0,1,1,1,1)\rightarrow
(1,1,0,0,0,1)\rightarrow
(0,0,1,0,1,0)\rightarrow
(1,1,0,1,0,0)\rightarrow
(0,0,1,0,0,1)\rightarrow
(1,1,0,1,1,1)\rightarrow
(0,1,1,0,0,0)\rightarrow
(1,0,0,1,1,0)\rightarrow
(0,1,1,0,1,1)\rightarrow
(1,0,0,1,0,1)\rightarrow
(0,1,1,1,1,0)\rightarrow
(1,0,0,0,0,0)\rightarrow
(0,1,1,1,0,1)\rightarrow
(1,0,0,0,1,1)\rightarrow
(0,1,0,1,0,0)\rightarrow
(1,0,1,0,1,0)\rightarrow
(0,1,0,1,1,1)\rightarrow
(1,0,1,0,0,1)\rightarrow
(0,1,0,0,1,0)\rightarrow
(1,0,1,1,0,0)\rightarrow
(0,1,0,0,0,1)\rightarrow
(1,0,1,1,1,1)\rightarrow
(1,1,0,0,0,0)\rightarrow
(0,0,1,1,1,0)\rightarrow
(1,1,0,0,1,1)\rightarrow
(0,0,1,1,0,1)\rightarrow
(1,1,0,1,1,0)\rightarrow
(0,0,1,0,0,0)\rightarrow
(1,1,0,1,0,1)\rightarrow
(0,0,1,0,1,1)\rightarrow
(1,1,1,1,0,0)\rightarrow
(0,0,0,0,1,0)\rightarrow
(1,1,1,1,1,1)\rightarrow
(0,0,0,0,0,1)\rightarrow
(1,1,1,0,1,0)\rightarrow
(0,0,0,1,0,0)\rightarrow
(1,1,1,0,0,1)\rightarrow
(0,0,0,1,1,1)\rightarrow
(1,0,1,0,0,0)\rightarrow
(0,1,0,1,1,0)\rightarrow
(1,0,1,0,1,1)\rightarrow
(0,1,0,1,0,1)\rightarrow
(1,0,1,1,1,0)\rightarrow
(0,1,0,0,0,0)\rightarrow
(1,0,1,1,0,1)\rightarrow
(0,1,0,0,1,1)\rightarrow
(1,0,0,1,0,0)\rightarrow
(0,1,1,0,1,0)\rightarrow
(1,0,0,1,1,1)\rightarrow
(0,1,1,0,0,1)\rightarrow
(1,0,0,0,1,0)\rightarrow
(0,1,1,1,0,0)\rightarrow
(1,0,0,0,0,1)\rightarrow
(0,1,1,1,1,1)$ represents a closed Euclidean knight's tour on $C(2,6)$, since the distance between its arrival and the starting vertex, ${\rm{V}}_{0}\equiv (0,0,0,0,0,0,0)$, is $d({{\rm{V}}}_{63}, {\rm{V}}_{0})=d((0,1,1,1,1,1),(0,0,0,0,0,0))=\sqrt{0^2+1^2+1^2+1^2+1^2+1^2}=\sqrt{5}$ (see Equation \ref{eq:1}).

That being so, we have constructed a closed Euclidean knight's tour on $C(2,6)$, the set of the $2^6$ corners of a $6$-cube.

Now, we note that any vertex of a $6$-face belonging to a unit $7$-cube is connected to some other vertices of the opposite $6$-face of the same $7$-cube by as many diagonals which are shorter than the $7$-cube diameter, including those characterized by a Euclidean length equal to the diameter of a unit $5$-cube (since $5<7$, trivially).

Thus, we can simply take the solution for the $k=6$ case and reproduce it also on the opposite $6$-face of the mentioned $7$-cube, and then we are free to mirror/rotate it so that the endpoints of both the covering paths of the two $6$-faces are connected by a pair of non-main diagonals of the $7$-cube with (Euclidean) length $\sqrt{5}$.

As an example, we can extend the $k=6$ solution $P_c{(2,6)}=(0,0,0,0,0,0) \rightarrow (1,1,1,1,1,0) \rightarrow \cdots \rightarrow (0,1,1,1,1,1)$ to $k=7$ as follows.

\begin{enumerate}
\item First of all, we move $P_c(2,6)$ from $(\mathbb{R}^6, d)$ to $(\mathbb{R}^7, d)$, and then we duplicate it as $S_1(2,7)=\{(0,0,0,0,0,0,0) \rightarrow (1,1,1,1,1,0,0) \rightarrow \cdots \rightarrow(0,1,1,1,1,1,0)\}$ and $S_2(2,7)=\{(0,0,0,0,0,0,1)\rightarrow (1,1,1,1,1,0,1) \rightarrow \cdots \rightarrow(0,1,1,1,1,1,1)\}$.
\item We apply $S_1(2,7)$ to the first $6$-face of $C(2,7)$ as it is (i.e., $\{(0,0,0,0,0,0,0) \rightarrow (1,1,1,1,1,0,0) \rightarrow \cdots \rightarrow(0,1,1,1,1,1,0)\}$), while we switch between $(0~\leftrightarrow~1)$ exactly $(5-1)$ more times out of $(7-1)$ Cartesian coordinates left (i.e., we modify $4$ other coordinates at our choice and we keep doing the switch $(0~\leftrightarrow~1)$ on the same, selected, coordinates of every node of $S_2(2,7)$, for the entire transformation of the aforementioned path) in order to place in a legit spot the arrival of the final path, since we can do this by simply reflecting/rotating and then reversing the whole polygonal chain $S_2(2,7)$ (e.g., $\{(0,1,1,1,1,0,1)\rightarrow (1,0,0,0,0,0,1) \rightarrow \cdots \rightarrow (0,0,0,0,0,1,1)\}$ indicates a valid geometric transformation of $S_2(2,7)$ that we can later reverse and finally apply on the proper $6$-face).
\item As a result, we get the Hamiltonian path for the opposite $6$-face of $C(2,7)$, $\hat{S_2}(2,7) \coloneqq \{(0,1,1,1,1,0,1)\leftarrow (1,0,0,0,0,0,1) \leftarrow \cdots \leftarrow (0,0,0,0,0,1,1)\}$, which is a polygonal chain whose starting point is at a distance of $$\sqrt{(0-0)^2+(0-1)^2+(0-1)^2+(0-1)^2+(0-1)^2+(1-1)^2+(1-0)^2}$$ from the ending point of $S_1(2,7)$ and whose arrival is, again, at a distance of $$\sqrt{(0-0)^2+(1-0)^2+(1-0)^2+(1-0)^2+(1-0)^2+(0-0)^2+(1-0)^2}$$ from the beginning of $S_1(2,7)$, letting our Euclidean knight jump onto $\hat{S_2}(2,7)$ at the end of $S_1(2,7)$ and vice versa, over and over.
\item Accordingly, we join the ending point of $S_1(2,7)$ and the starting point of $\hat{S_2}(2,7)$ to get the closed Euclidean knight's tour described by $S_1(2,7) \cup \{(0,1,1,1,1,1,0) \rightarrow (0,0,0,0,0,1,1)\} \cup \hat{S_2}(2,7)$, a Hamiltonian path $\{(0,0,0,0,0,0,0)\rightarrow \cdots \rightarrow (0,1,1,1,1,0,1)\}$ for $C(2,7)$ that, without loss of generality, we have constructed from a closed Euclidean knight's tour on $C(2,6)$ beginning from $(0,0,0,0,0,0) \in \mathbb{Z}^6$.
\end{enumerate}

Finally, we can repeat the same process to extend the $7$-cube solution to the $8$-cube, and so forth. We are allowed to do so since any vertex of a $k$-cube is also a corner, and thus the symmetry is preserved.

Therefore, there exist (many) closed Euclidean knight's tours on each $2 \times 2 \times \cdots \times 2$ chessboard with at least $2^6$ cells, whereas no Euclidean knight's tour is possible on $C(2,k)$ for $k < 6$, and the proof of Theorem~\ref{Theorem 4.1} is complete.
\end{proof}


\section{Conclusion} \label{sec:Conc}

The move pattern of the $2$D knight, the piece that also appears in the FIDE logo, is described by Article~3.6 of Reference \cite{1} as follows: ``\textit{The knight may move to one of the squares {\bf{\textit{nearest}}} to that on which it stands but not on the same rank, file or diagonal}''. Thus, we have defined the knight by assuming the standard Euclidean metric and, consequently, the resulting distance covered by any jump is $\sqrt{5}$.

Then, it seems legitimate to assume that a multidimensional knight can also be defined in a more inclusive way than the usual description of a piece that merely moves, from a vertex of $C(n,k)$ to another one of the same set, by adding or subtracting $2$ to one of the $k$ Cartesian coordinates of the starting vertex and, simultaneously, adding or subtracting $1$ to another of the remaining $k-1$ elements of the mentioned $k$-tuple. Accordingly, we have provided the alternative definition of the knight as a piece that is allowed to move from ${{\rm{V}}}_{m} \in C(n,k)$ to ${{\rm{V}}}_{m+1} \in C(n,k)$ if and only if the condition $d({{\rm{V}}}_{m}, {{\rm{V}}}_{m+1})=d({{\rm{V}}}_{m+1}, {{\rm{V}}}_{m})=\sqrt{5}$ is satisfied (where $d({{\rm{A}}}, {{\rm{B}}})$ indicates the Euclidean distance between the point ${{\rm{A}}}$ and the point ${{\rm{B}}}$, as usual).

Consequently, the present paper has shown that such a Euclidean knight can produce a knight's tour also on $k$-dimensional grids of the form $\{\{0,1,2\} \times \{0,1,2\}\times \cdots \times \{0,1,2\}\}$, for some $k \geq 5$. In particular, by Corollary~\ref{Corollary 2.3}, every Euclidean knight's tour for any $3 \times 3 \times \cdots \times 3$ chessboard is necessarily an open tour.

Now, if this is not enough to fully accept the $\sqrt{5}$ knight metric, since the knight may not be allowed to move along one of the main diagonals of a $5$-cube by unimaginatively extending beyond the $k=2$ case the semantic meaning of Article~3.6 of \cite{1}, then Theorem~\ref{Theorem 4.1} shows that it is possible to construct closed Euclidean knight's tours on any ($k$-dimensional) $2 \times 2 \times \cdots \times 2$ chessboard, as long as $k \geq 6$, avoiding by construction to move along any main diagonal of the given $k$-cube. Furthermore, we could invoke the same argument to suggest that, for any $k > 4$, the central vertex of a $3 \times 3 \times \cdots \times 3$ $k$-cube is (Euclidean) knight-connected to any other element of the set $C(3,k)$.

As a result, Theorem~3 of \cite{17} can no longer be invoked on any metric space $C(2,k) : k > 5$ where the distance between two vertices ${{\rm{A}}} \in C(n,k)$ and ${{\rm{B}}} \in C(n,k)$ is given by the minimum number of jumps \cite{21} of (Euclidean) length $\sqrt{5}$ that the described Euclidean $k$-knight requires in order to go from ${{\rm{A}}}$ to ${{\rm{B}}}$ (and vice versa).

Lastly, a related open problem is to prove the existence of a closed (possibly Euclidean) knight's tour also for every $C(3,k)-\{(1,1,\ldots,1)\}$ such that $k$ is even (in Section~\ref{sec:3}, we have verified the correctness of this conjecture for the cases $k=2$ and $k=4$, but making such inferences in low dimensions can be notoriously misleading).

\section*{Acknowledgements} 

The author sincerely thanks Aldo Roberto Pessolano for his invaluable effort, greatly helping him to prove Theorem~4.1 by performing a computational analysis that has returned, in just a few seconds, independent closed Euclidean knight's tours for $2 \times 2 \times \cdots \times 2$ chessboards of $2^7$, $2^8$, $2^9$, $2^{10}$, $2^{11}$, and $2^{12}$ cells (we have not included them in the present paper to save space). Additionally, the author is grateful to the Mathematics Stack Exchange users Steven Stadnicki and Ben Grossmann for having suggested the existence of closed Euclidean knight's tours also for the $2 \times 2 \times 2 \times 2 \times 2 \times 2$ chessboard.

\makeatletter
\renewcommand{\@biblabel}[1]{[#1]\hfill}
\makeatother

\end{document}